\documentclass[submission,copyright,creativecommons]{eptcs}
\providecommand{\event}{ACT 2025} 

\usepackage{iftex}

\ifpdf
  \usepackage{underscore}         
  \usepackage[T1]{fontenc}        
\else
  \usepackage{breakurl}           
\fi

\usepackage[final]{microtype}
\usepackage{amsthm,mathtools,quiver}
\usepackage{xcolor}
\newtheorem{theorem}{Theorem}[section]
\newtheorem{lemma}[theorem]{Lemma}

\newtheorem{proposition}[theorem]{Proposition}

\newtheorem{corollary}[theorem]{Corollary}

\theoremstyle{definition}
\newtheorem{definition}[theorem]{Definition}

\usepackage{newtxmath,newtxtext}
\usepackage[mode=buildmissing]{standalone}

\newcommand{\Set}{\mathsf{Set}}
\title{Convex Duality made Difficult}
\author{Eigil Fjeldgren Rischel\institute{Tallinn University of Technology}\footnote{This work was supported by the ARIA programme on Safeguarded AI}}
\begin{document}

\def\authorrunning{Eigil Fjeldgren Rischel}
\def\titlerunning{Convex Duality made Difficult}

\maketitle
\section{Introduction}\par{}The study of convex functions - in particular, of their optimization (really \emph{minimization}) is one of the most important fields of applied mathematics. Convexity seems to be one of those incredibly well-chosen hypotheses which is just specific enough to admit a wealth of theorems, just general enough to produce a nontrivial theory (and a large amount of important examples).

Convex optimization, possibly because it has an "analytical" rather than "algebraic" feel, has not been very thoroughly studied by applied category theorists. The one notable exception is \cite{hanks-etal-convex-2024}, which studies the decomposition of optimization problems by categorical means. This paper takes a different approach, attempting to define a category with optimization problems as the objects, and to prove theorems about optimization by categorical means. As an illustration, we show how to use our methods to rederive some existing results: A minimax-type theorem, Theorem~\ref{efr-001E}, and the fact that for convex functions, $(f^*)^* = f$ (where $f$ is the Legendre dual), Proposition~\ref{efr-001V}. 

\section{Convex optimization}

Let us first recall some of the basic objects of convex optimization, taking this opportunity to fix our terminology. See \cite{boyd-vandenberghe-convexopt-2004} for a standard reference on this.

\begin{definition}[{Standard form convex optimization problem}]\label{efr-001G}
  
  A \emph{convex optimization problem in standard form} consists of
    \begin{enumerate}\item{}A convex function \(f_0: \mathbb {R}^k \to  \mathbb {R}\)
        \item{}A list of convex functions \(f_1, \dots  f_n: \mathbb {R}^k \to  \mathbb {R}\)
        \item{}A list of \emph{affine} functions \(g_1,\dots  g_m: \mathbb {R}^k \to  \mathbb {R}\)\end{enumerate}
    The problem then is to find \(x \in  \mathbb {R}^k\) which minimizes \(f_0(x)\) subject to the constraints \(f_i(x) \leq  0, g_i(x)=0\).

    A \emph{concave} optimization problem is one where instead the function $f$ to be optimized is concave, and the goal is maximization (note that we still optimize on a convex set).

In the following, we will take $\mathbb{R}_+$, somewhat unconventionally, to denote the \emph{nonnegative} reals. $\mathbb{R}^n_+$ simply denotes the set of vectors which is nonnegative in each coordinate.

\end{definition}

\begin{definition}[{Lagrangian of an optimization problem}]\label{efr-002K}\par{}Let \(f_0: \mathbb {R}^k \to  \mathbb {R}, f_1, \dots  f_n, g_1 , \dots  g_m\) be a standard-form convex optimization problem, as in Definition~\ref{efr-001G}. Then the \emph{Lagrangian} of this problem is the function \(L: \mathbb {R}^k \times  \mathbb {R}^n_+ \times  \mathbb {R}^m \to  \mathbb {R}\) defined by
  
  \[L(x;\lambda ,\nu ) = f_0(x) + \sum _i \lambda _i f_i(x) + \sum _i \nu _i g_i(x).\]
  \end{definition}\par{}Observe that \(\sup _{\lambda ,\nu } L(x,\lambda ,\nu )\) is \(f_0(x)\) if \(x\) satisfies the constraints of the problem, and \(\infty \) otherwise. Hence we can think of this minimization problem as playing a zero-sum game: we choose \(x\), our adversary chooses \(\lambda ,\nu \), and our loss function is \(L\).\par{}It is natural to ask about the existence of (Nash) equilibria in this game - observe that the existence of an equilibrium \((x^*,\lambda ^*,\nu ^*)\) means that \(\inf _x \sup _{\lambda ,\nu } L(x,\lambda ,\nu ) = \sup _{\lambda ,\nu } \inf _x L(x,\lambda ,\nu ) = L(x^*,\lambda ^*,\nu ^*)\).
  
  This is of great utility in solving the original problem.\par{}
  
  The \emph{dual problem} is the problem of \emph{maximizing} the function \(\inf _x L(x,\lambda ,\nu )\).
  This is always a concave problem (in the sense that the infimum is a concave function of $\lambda, \nu$).

  In the world of convex optimization, two problems whose constraints carve out the same subset of \(\mathbb {R}^k\) (and where the function to optimize is the same) would be called \emph{equivalent}. But they can clearly not be regarded as \emph{isomorphic}, because the choice of constraint functions makes an important difference to the theory of optimization (for example, it can lead to different dual problems). Here we take the viewpoint that the \emph{Lagrangian} is really the fundamental object in convex optimization - by passing to a suitable category of Lagrangians, we can make the dual problem into an actual self-duality on this category.
  
  \section{Convex spaces}\begin{definition}[{Convex Space}]\label{efr-0020}\par{}The category of \emph{convex spaces} is the category of algebras for the monad \(\Delta : \mathsf {Set} \to  \mathsf {Set}\) of discrete finite-support distributions. The morphisms are called \emph{\(\Delta \)-homomorphisms} or \emph{homomorphisms of convex spaces}.\par{}So as not to multiply notation unnecessarily, we simply denote the category of convex spaces by \(\mathsf {Set}^\Delta \), using the usual notation for the Eilenberg-Moore category.\end{definition}
  
  Convex spaces have been studied many times - see eg \cite{fritz2009convex} for a systematic description of $\Set^\Delta$.  

  \begin{definition}[{}]\label{efr-002N}\par{}A function between vector spaces is called \emph{affine} if it preserves those linear combinations \(\sum _i \lambda _i x_i\) where \(\sum _i \lambda _i = 1\)\end{definition}
  
  \begin{definition}[{}]\label{efr-002M}\par{}Let \(X\) be a convex space. A \emph{convex function on \(X\)} is a function \(f: X \to  \mathbb {R}\) so that
  \[f(\theta  x + (1-\theta )x') \leq  \theta  f(x) + (1-\theta )f(x')\]\par{}A \emph{concave function} is a function so that \(-f\) is convex (in other words, \(f\) satisfies the opposite inequality).\end{definition}\par{}The term "convex function" in this sense clashes with the usual practice of naming structure-preserving functions after the structure they preserve (since convex functions do not preserve the convex structure). Unfortunately this usage is far too established to alter. (Convex functions are called convex because they are exactly those functions where the area above their graph is a convex subset of \(X \times  \mathbb {R}\). Although there appears to be no particular reason why the terms convex and concave should not be interchanged, other than convention).
  \par{}The inequality
  \[f(\theta  x + (1-\theta )x') \leq  \theta  f(x) + (1-\theta )f(x'),\]
  which holds whenever \(f\) is convex, is called \emph{Jensen's inequality}.
  Sometimes this name is used for a stronger version of this inequality,
  like the claim that \(f(\mathbb {E} X) \leq  \mathbb {E} f(X)\) if \(X\) is a random variable valued in the domain of \(f\). These generally follow just from convexity of \(f\).\label{efr-002Q}\par{}There is an natural way to extend the convex structure of \(\mathbb {R}\) to both \([ -\infty , \infty  )\) and \(( -\infty , \infty  ]\), by the convention that any nontrivial convex combination involving an infinity is equal to that infinity. This also gives the adjectives convex and concave a meaning when applied to functions \(X \to  ( -\infty , \infty  ]\). For example, a function \(f: X \to  ( -\infty , \infty  ]\) is convex if and only if the subset where it's finite is a convex subset of \(X\), and it's a convex function in the ordinary sense on this set.\par{}This doesn't work for the extended real line \(\overline {\mathbb {R}} = [ -\infty , \infty  ]\), since there is no sensible interpretation of \(\theta  \cdot  -\infty  + (1-\theta )\infty \). We will inescapably meet some functions which take value in the full extended reals, but where we still wish to speak of their convexity (or concavity).\par{}Hence we adopt the convention that a function \(f: X \to  \overline {\mathbb {R}}\) is convex if it obeys Jensen's inequality whenever it makes sense, i.e whenever we do not have \(f(x) = -\infty , f(x') = \infty \) or vice versa.\begin{proposition}[{}]\label{efr-002O}\par{}A function between vector spaces is \hyperref[efr-002N]{affine} if and only if it is a \hyperref[efr-0020]{\(\Delta \)-homomorphism}.\begin{proof}\par{}It's clear that an affine function is a \(\Delta \)-homomorphism.
    Suppose \(f: X \to  Y\) is a \(\Delta \)-homomorphism. Note it suffices to prove \(f\) preserves binary affine combinations \(\theta  x + (1-\theta )x'\) (for \(\theta \) not necessarily in \([0,1]\)).
    If \(\theta  \in  [0,1]\), we are done by assumption. Otherwise suppose \(\theta  > 1\) (if not, replace it by \(1-\theta \) by symmetry). Then
    \[x = (1/\theta )(\theta  x + (1-\theta )x') + (1 - 1/\theta )x'\]
    This is a convex combination, so
    \[f(x) = (1/\theta )f(\theta  x + (1-\theta )x') + (1-1/\theta )f(x')\]
    Rearranging, we find
    \[\theta  f(x) + (1-\theta )f(x') = f(\theta  x + (1-\theta ')x)\]
    as desired.\end{proof}\end{proposition}\par{}Justified by Proposition~\ref{efr-002O}, we will appropriate the term \emph{affine} to refer to \(\Delta \)-homomorphisms, even between convex spaces which are not vector spaces. There is generally no chance of confusion, but it's worth emphasizing that the use of this term does not entail that the domain is closed under arbitrary affine combinations, for example.
    
    Convex spaces admit both a Cartesian product (given by the product of the underlying sets equipped with pointwise operations) and a tensor product, which (co)represents "bihomomorphisms". This is analogous to the situation for vector spaces. Unlike vector spaces, however, since all constant maps are homomorphisms, the projections \(X \times  Y \to  X,Y\) are bihomomorphisms, which induces a map \(X \otimes  Y \to  X \times  Y\). Thus homomorphisms \(X \times  Y \to  Z\) are a subset of bihomomorphisms.\begin{definition}[{Simplex}]\label{efr-002S}\par{}The free convex space on a finite set \(\{0, \dots  n\}\) of \(n+1\) elements is called the \emph{\(n\)-simplex} and denoted \(\Delta ^n\) (the reason for the apparent mismatch of numbering is that the \(n\)-simplex is \(n\)-dimensional). Note that an element of \(\Delta ^n\) is a tuple \((s_i)_{i=0,\dots , n}\) so that \(\sum _i s_i = 1\) and \(s_i \geq  0\).
  In particular, \(\Delta ^1 \cong  [0,1]\).\end{definition}
  \begin{definition}[{Topological convex space}]\label{efr-002U}\par{}A \emph{topological convex space} is a convex space \(X\) equipped with a topology so that any affine map \(\Delta ^n \to  X\) is continuous (when \(\Delta ^n \subseteq  \mathbb {R}^{n+1}\) is given the subspace topology).\end{definition}

  \section{The Category of Minmax problems}
  \begin{definition}[{Minmax problem}]\label{efr-001C}\par{}A \emph{minmax problem} is a triple \((X,Y,L),\) where \(X,Y\) are convex spaces, and \(L: X \times  Y \to  \mathbb {R}\) is a function which is
    \begin{enumerate}\item{}Pointwise \emph{convex} in \(X\) - for each \(y\), given \(x_1,x_2 \in  X, \theta  \in  [0,1]\), \[L(\theta  x_1 + (1-\theta )x_2,y) \leq  \theta  L(x_1,y) + (1-\theta )L(x_2,y)\]
        \item{}Pointwise \emph{concave} in \(Y\) - for each \(x\), given \(y_1,y_2 \in  Y, \theta  \in  [0,1]\), \[L(x,\theta  y_1 + (1-\theta )y_2) \geq  \theta  L(x,y_1) + (1-\theta )L(x,y_2)\]\end{enumerate}\par{}A \emph{morphism of minmax problems} \((X,Y,L) \to  (X',Y',L')\) is a pair of functions \(\phi ^+: X \to  X'\) and \(\phi ^-: Y' \to  Y\) so that \(L(x, \phi ^-(y')) \geq  L'(\phi (x),y')\)
    \end{definition}
        
    We will see that various constructions on this category, which are natural and well-behaved from the point of view of category theory, capture relevant constructions from the theory of convex optimization.\begin{enumerate}\item{}\(\mathsf {Minmax}\) is bifibred over \(\mathsf {Set}^\Delta  \times  \mathsf {Set}^{\Delta ,\mathrm {op}}\), and the Cartesian and coCartesian lifts capture the operations of minimizing over the primal variables or maximizing over the dual variables
    \item{}The property of \emph{strong duality} amounts to the claim that a particular diagram has the local Beck-Chevalley property
    \item{}Relatedly, the existence of a Nash equilibrium for the game corresponding to \(L\) amounts to the existence of a certain morphism. The fact that this implies strong duality can be derived by purely categorical means.\end{enumerate}\begin{proposition}[{}]\label{efr-002H}\par{}Let \((X,A,L)\) be a minmax problem.
  Suppose \(A\) is a convex subspace of a vector space \(V\), and \(L(x,-): A \to  \mathbb {R}\) is affine for each \(x\). Suppose $A$ contains an open subset of $V$.
  Then there exists unique convex functions \(f: X \to  \mathbb {R}, g: X \to  V^*\), so that \(L(x,a) = f(x)+\langle  g(x),a \rangle \).\end{proposition}\par{}Observe that minmax problems affine in \(A\) are thus very similar to standard-form convex optimization problems, the main difference being that the set of allowed points in \(A\) may be constrained in some other way than by requiring certain coordinates to be nonnegative.\par{}
  We omit the proof for brevity, but note that the key point is that an affine function on $A$ always admits a unique extension to $V$, which is always given by a linear function plus a constant.  
  \begin{definition}[{Primal and dual optimization problems}]\label{efr-001D}\par{}Let \(L: X \times  Y \to  \mathbb {R}\) be a \hyperref[efr-001C]{minimax problem}.
    The \emph{primal optimization problem} associated to \(L\) is the function
    \[L^+(-) = \sup _y L(-,y): X \to  \mathbb {R}\]
    (the problem being to \emph{minimize} this function).\par{}The dual optimization problem is the function \(L^-(-) = \inf _x L(x,-): Y \to  \mathbb {R}\)\end{definition}\begin{definition}[{Dual minmax problem}]\label{efr-001J}\par{}Let \(L = (X,Y,L)\) be a \hyperref[efr-001C]{minmax problem}. Then let \(L^*\) denote the \emph{dual} problem given by \((Y,X,L^*(y,x) = -L(x,y))\).\par{}If \(\phi  = (\phi ^+,\phi ^-) : L \to  L'\) is a morphism of minmax problems, then \(\phi ^* = (\phi ^-,\phi ^+): L'^* \to  L^*\) is again a morphism in the other direction. This assignment makes \((-)^*\) into a self-inverse functor on the category of minmax problems\end{definition}\par{}We will often utilize this duality to abbreviate proofs, proving something, for example, for the forwards direction and arguing "by duality" that it holds for the backwards direction as well.\begin{definition}[{Backwards and forwards morphisms}]\label{efr-0027}\par{} Let a morphism \(\phi  = (\phi ^+,\phi ^-)\) in \(\mathsf {Set}^\Delta  \times  \mathsf {Set}^{\Delta ,\mathrm {op}}\) be called \emph{forwards} if \(\phi ^-\) is an isomorphism, and \emph{backwards} if \(\phi ^+\) is an isomorphism.
    \end{definition}
      Let \(F\) denote the set of forwards morphisms, \(B\) the set of backwards. Then clearly \((F,B)\) form an orthogonal factorization system - in fact, both \((F,B)\) and \((B,F)\) do.\par{}We will say a morphism in \(\mathsf {Minmax}\) is forwards, respectively backwards, if it is so considered as a morphism in \(\mathsf {Set}^\Delta  \times  \mathsf {Set}^{\Delta ,\mathrm {op}}\), and reuse the notation \(F,B\) for these subclasses of morphism.
  
  \begin{lemma}[{}]\label{efr-002A}\par{}Let \(X,Y\) be convex spaces and let \(A \subset  X \times  Y\) be a convex subspace. Let \(f: A \to  \mathbb {R}\) be a convex function.
  Then \(x \mapsto  \inf _{y: (x,y) \in  A} f(x,y)\) is again convex.\begin{proof}\par{}Let \(\theta  \in  [0,1],x,x' \in  X\) be given, and consider:\[\inf_{y: (\theta  x + (1-\theta )x',y) \in  A}f(\theta  x + (1-\theta  x'),y).\]\par{}Since if \((x,y), (x',y') \in  A\) then \((\theta  x + (1-\theta )x',\theta  y + (1-\theta )y') \in  A\), we have that this is less than:
  \[\leq  \inf _{y,y': (x,y),(x',y')\in  A} f(\theta  x + (1-\theta )x', \theta  y + (1-\theta ) y'),\]
  because in the latter we are taking the infimum over a smaller set of \(f\)'s\par{}Applying convexity, we get
    \[\leq  \inf _{y,y': (x,y),(x,y') \in  A} \theta  f(x,y) + (1-\theta )f(x',y')\]
    \[\leq  \theta  \inf _{y: (x,y) \in  A}f(x,y) + (1-\theta )\inf _{y': (x',y')\in  A} f(x',y')\]
    This is precisely the desired inequality.\end{proof}\end{lemma}
    
    \begin{proposition}[{}]\label{efr-0023}\par{}The forgetful functor \(\mathsf {Minmax} \to  \mathsf {Set}^\Delta  \times  \mathsf {Set}^{\Delta ,\mathrm {op}}\) is a bifibration. Moreover, we have the following description of the (co)Cartesian morphisms over backwards and forwards maps.
  \begin{enumerate}\item{}A forwards morphism \((\phi ,1_A): (X,A,L) \to  (Y,A,L')\) is Cartesian if and only if \(L(x,a) = L'(\phi (x),a)\) for all \(x,a\)
  \item{}A forwards morphism \((\phi ,1_A): (X,A,L) \to  (Y,A,L')\) is coCartesian if and only if \(L'(y,a) = \inf _{\phi (x) = y} L(x,a)\)
  \item{}A backwards morphism \((1_X,\phi ): (X,A,L) \to  (X,B,L')\) is Cartesian if and only if \(L(x,a) = \inf _{\phi (b)=a} L'(x,b)\)
  \item{}A backwards morphism \((1_X,\phi ): (X,A,L) \to  (X,B,L')\) is coCartesian if and only if \(L'(x,b) = L(x,\phi (b))\)\end{enumerate}\begin{proof}\par{}Note that it suffices to provide Cartesian and coCartesian lifts for backwards and forwards morphisms (Definition~\ref{efr-0027}), since such lifts compose. Hence it suffices to verify that the given descriptions are correct, since clearly they suffice to compute a (co)Cartesian lift over any such morphism.\par{}Note also that, since the forgetful functor is faithful, to verify a morphism \(\phi \) is (co)Cartesian, it suffices to prove that any factorization in the base lifts - uniqueness is automatic.\par{}Thus let \(\phi  = (\phi ,1_A): (X,A,L) \to  (Y,A,L')\) be so that \(L(x,a) = L'(\phi (x),a)\). Note that composition of a \(\Delta \)-homomorphism with a convex function is again convex, so this is indeed an object of \(\mathsf {Minmax}\)\par{}Now let \(\psi  = (\psi ^-,\psi ^+): (Z,B,K) \to  (Y,A,L')\) be some morphism so that we have the factorization \(\psi  = \phi \psi '\) in \(\mathsf {Set}^\Delta \times \mathsf {Set}^{\Delta ,\mathrm {op}}\). The goal is now to prove \(\psi ' : (Z,B,K) \to  (X,A,L)\) is a homomorphism. This is the inequality
  \[K(z,(\psi ')^-(a)) \geq  L((\psi ')^+(z),a) = L'((\phi \psi ')^+(z),a),\] which holds by assumption\par{}Let \(\phi \) be as above, but suppose \(L'(y,a) = \inf {\phi (x)=y}L(x,a)\). First, observe that by Lemma~\ref{efr-002A}, this function is in fact convex in \(y\) as desired.\par{}Let \(\psi : (X,A,L) \to  (Z,B,K)\) be given, and now suppose we have a factorization \(\psi  = \psi '\phi \) in the base. We must prove that \(L'(y,(\psi ')^-(b)) \geq  K((\psi ')^+(y),b),\)
    but since \(L'(y,(\psi ')^-(b)) = \inf _{\phi (x)=y}L(x,(\psi ')^-(b)),\) this amounts to the equation
    \(L(x,(\psi ')^-(b)) \geq  K((\psi ')^+\phi (x),b),\)
    which is again true by assumption.\par{}Now the case for backwards morphisms simply follows by duality.\end{proof}\end{proposition}\par{}What's "really" going on here is that \(\mathsf {Minmax}\) is a two-sided fibration, the result of taking the functor \(\mathsf {Set}^{\Delta ,\mathrm {op}} \times  \mathsf {Set}^{\Delta ,\mathrm {op}} \to  \mathsf {Cat}\) carrying a pair \(X,Y\) to the poset of minmax problems \(L: X \times  Y \to  \mathbb {R}\) (in the opposite order), with morphisms acting by precomposition, and applying the Grothendieck construction "contravariantly in the first variable and covariantly in the second variable". (And then observing that the precomposition action has left/right adjoints given by \(\inf \)/\(\sup \), to make this into a \emph{bi}fibration). But the theory of two-sided fibrations is quite complicated in general, and we will not go into it here - see \cite{Loregian2020} for a definition and basic properties.\par{}Note also that this functor is quite close to displaying \(\mathsf {Minmax}\) as \emph{topological} (see \cite{dubuc2006topological}). If we remove the restriction that minmax problems be convex/concave, we can construct the universal lifts required using a similar supremum formula. The problem is that the supremum of a general set of concave functions is not automatically concave (however, the supremum taken over a convex set, in a suitable sense, is).\begin{definition}[{\(\mathsf {Conv}\) and \(\mathsf {Conc}\)}]\label{efr-001R}
      \begin{enumerate}
      \item Let \(\mathsf {Conv}\) be the category where objects are pairs \((X,f: X \to  \mathbb {R})\) consisting of a convex space and a convex function, and where morphisms \(\phi : (X,f) \to  (Y,g)\) are affine maps so that \(g(\phi (x)) \leq  f(x)\).\par{}\item Let \(\mathsf {Conc}\) be the category where objects are pairs \((X,f: X \to  \mathbb {R})\) consisting of a convex space and a concave function, and where morphisms \(\phi : (X,f) \to  (Y,g)\) are affine maps so that \(g(\phi (x)) \geq  f(x)\).
      \end{enumerate}
    \end{definition}
    
    \begin{proposition}[{}]\label{efr-001S}\par{}
      \begin{enumerate}
      \item The assignment
    \[(X,Y,L) \mapsto  (X,L^+),\]
    \[(\phi ^+,\phi ^-): L \to  L' \mapsto  \phi ^+\]
    defines a functor \((-)^+: \mathsf {Minmax} \to  \mathsf {Conv}\)
    \item Similarly, \((-)^-\) defines a functor \(\mathsf {Minmax} \to  \mathsf {Conc}^\mathrm {op}\). (The reason for this idiosyncratic way of writing a contravariant functor will become apparent in a minute)
    \item The assignment \((X,f) \mapsto  (X,-f)\) defines a functor (an isomorphism of categories) \(\mathsf {Conc} \to  \mathsf {Conv}\), and vice versa. Then \(L^- = -(L^*)^+\)
    \item The assignment \((X,f) \mapsto  (X,*, f)\) defines a fully faithful functor \(\mathsf {Conv} \to  \mathsf {Minmax}\), whose essential image consists of those tuples \((X,Y,L)\) where \(Y\) is singleton.
    \item Analogously, \((Y,f) \mapsto  (*,Y,f)\) defines a fully faithful functor \(\mathsf {Conc}^\mathrm {op} \to  \mathsf {Minmax}\)
    \item We will abuse notation and identify \(\mathsf {Conv}\) and \(\mathsf {Conc}\) with their images under these inclusions - thus, for example, \(L^+\) will be regarded as an object of \(\mathsf {Minmax}\).
    \item \((-)^+\) is right adjoint to the inclusion of \(\mathsf {Conv}\), and \((-)^-\) (viewed as a functor \(\mathsf {Minmax} \to  \mathsf {Conc}^\mathrm {op}\)) is left adjoint to the inclusion of \(\mathsf {Conc}^\mathrm {op}\)
    \item Using these identifications, we have \((-)^- = (((-)^*)^+)^*\)
      \end{enumerate}
  \end{proposition}\par{}Note that if \(\phi  = (\phi ^+,\phi ^-): L \to  L'\) is a morphism of \(\mathsf {Minmax}\), the two meanings of the notation \(\phi ^+\) agree, and the same is true of \(\phi ^-\).\par{}Note also that the reflexive subcategory \(\mathsf {Conc}^\mathrm {op} \subseteq  \mathsf {Minmax}\) is the local subcategory with respect to the forwards morphisms - a morphism is forward if and only if \(\phi ^-\) is an isomorhism (by definition), and the unit \(L \to  L^-\) is the terminal forwards morphism with domain \(L\). A dual statement holds for \(\mathsf {Conv} \subseteq  \mathsf {Minmax}\) (it is the \emph{colocalization} with respect to the class of backwards morphisms).
    \begin{definition}[{Monoidal structure on minmax problems}]\label{efr-001N}\par{}There is a monoidal structure on minmax problems, given by
    \((L \otimes  L') = (X \otimes  X', Y \otimes  Y', (x,x',y,y') \mapsto  L(x,y) + L(x',y')).\)
    The unit is \((*,*,0)\). \end{definition}\par{}A state (that is, a morphism $I \to (L,X,A)$) is a point \(x_0\) so that \(L(x_0,y) \leq  0\) for all \(y\).
    More interesting is asking for a state of \(L \otimes  L^*\).
    This is a pair \(x \in  X, y \in  Y\) so that the inequality
    \(L(x,y') \leq  L(x',y)\) holds for all \(y',x'\)\par{}Note that \(\sup _{y'} L(x,y') \geq  \inf _x L(x',y)\) for all \(x,y\), this is the minmax inequality (or "weak duality").\par{}Thus a choice of \(x,y\) giving a state gives equality in that inequation - it is a \emph{solution} of the minmax game.By duality, and since \((L \otimes  L^*)^* \cong  L \otimes  L^*\), states and costates are in bijection for such an object.\begin{proposition}[{}]\label{efr-002D}\par{}The forgetful functor \(\mathsf {Minmax} \to  \mathsf {Set}^\Delta  \times  \mathsf {Set}^{\Delta ,\mathrm {op}}\) is a monoidal fibration, in the sense of \cite{shulman-monfib-2009}, (see also \cite{moeller-vasilakopoulou-2018}). It is also a monoidal opfibration. $\mathsf{Conv}, \mathsf{Conc}^\mathrm{op}$ are monoidal subcategories, and $(-)^+,(-)^-$ are strong monoidal functors for the restricted monoidal product.\end{proposition}
    
    \section{Strong duality}
    \begin{proposition}[{Weak duality}]\label{efr-001W}\par{}Let \(L\) be a minmax problem.
    Then
    \[\inf _x \sup _y L(x,y) = (L^+)^- \geq  (L^-)^+ = \sup _x \inf _y L(x,y),\]
    where we abuse notation by identifying a minmax problem \(*,*, r\) with the number \(r(*,*)\)
    \end{proposition}
    This is intuitively clear, but also follows from considering the adjunction properties of $(-)^+,(-)^-$.
        
    \begin{definition}[{Strong duality}]\label{efr-002F}\par{}Let \(L = (X,A,L)\) be a minmax problem. By Proposition~\ref{efr-001W}, there is a morphism
  \((L^+)^- \to  (L^-)^+\). We say \(L\) \emph{satisfies strong duality} if it is an isomorphism. (Note that this is really just an inequality of real numbers, which must be an equality).\end{definition}
    Observe that for a minmax problem arising as the Lagrangian of a convex optimization problem in standard form, this is precisely the classical notion of strong duality, see \cite[Section 5.2.3]{boyd-vandenberghe-convexopt-2004}

  \begin{proposition}[{}]\label{efr-001X}\par{}Let \(L\) be a minmax problem.
    Suppose there exists \(\phi : I \to  L \otimes  L^*\).
    Then strong duality holds, i.e \((L^+)^- \cong  (L^-)^+\). \(((\phi )^+)^-\) gives a morphism
    \[I = (I^+)^- \to  ((L \otimes  L^*)^+)^- \cong  (L^+)^- \otimes  ((L^*)^+)^- \cong  (L^+)^- \otimes  ((L^-)^+)^*\]
  \end{proposition}
    Here we use the isomorphisms \((L^+)^* = (L^*)^-\) and vice versa, as well as strong monoidality of \((-)^-\) and \((-)^+\).
    The existence of that morphism means that
    \((L^+)^- \leq  (L^-)^+,\)
    which is the other direction of the morphism we wanted.

    If a minmax problem is a zero-sum game, a point \(I \to  L \otimes  L^*\) is a choice of Nash equilibrium for this game.\begin{proposition}[{}]\label{efr-0029}\par{}Let \((X,Y,L)\) be a minmax problem.
  Then there is a canonical commutative diagram
  
    \begin{center}
        \begin {tikzcd}
            (X,*) \ar [r, "\pi _X"] \ar [d, "\pi _A"] & (*,*)\ar [d, "\pi _A"]\\
            (X,A) \ar [r, "\pi _X"] & (*,A)
            \end {tikzcd}
    \end{center}

  \noindent in \(\mathsf {Set}^\Delta  \times  \mathsf {Set}^{\Delta ,\mathrm {op}},\) which is a pullback.
  \(L\) obeys strong duality if and only if this square has the local Beck-Chevalley condition for \(L\), in the sense that the canonical map \(\pi _{X,!}\pi _A^*L \to  \pi _A^*\pi _{X,!}L\) is an isomorphism. (See \cite{pavlovic-descent-beck-chevalley-1991} for more on the Beck-Chevalley condition)
 
  \begin{proof}Recall that \(\pi _{X,!}(L) = L^-, \pi _Y^*(L) = L^+\). Hence the claim is just that \(\inf _x \sup _a L(x,a) = (L^+)^- = (L^-)^+ = \sup _a\inf _x L(x,a),\) which is precisely strong duality.\end{proof}\end{proposition}
  
  We now prove a minimax theorem for a class of our minimax problems. Note that it relies crucially on \emph{compactness}, and so doesn't apply to the Lagrangians of standard-form convex optimization problems.

  \begin{theorem}[{Minimax theorem}]\label{efr-001E}\par{}Let \((L,X,A) \in  \mathsf {Minmax}\). If \(X,A\) are both convex, compact subspaces of finite-dimensional vector spaces, and \(L\) is continuous, then strong duality holds for \(L\), and moreover an equilibrium \(I \to  L \otimes  L^*\) exists.
  \end{theorem}

  The original minimax theorem, due to Von Neumann (\cite{vonneumann1928}) is the special case where $X,A$ are both standard simplices $\Delta^n, \Delta^m,$ and $f$ is affine (not merely convex). This has been generalized many times, including some which have the above as a special case (see for example \cite{siongeneral1958}). The novelty here is not the theorem, but the categorical approach to the proof.
  
  Our theorem can be derived from the Kakutani fixpoint theorem in a very similar way to the usual proof of Nash's theorem about general, non-zerosum games - although note that it is not a special case, since \(X\) and \(A\) may not be simplices, and the payoff function here is merely convex, not necessarily \emph{affine} as it is for a game-theoretic game.\par{}However, we will give a different proof, which uses the structure of \(\mathsf {Minmax}\) in a more direct way. Essentially, we will use compactness to reduce to the case of simplexes, then use an inductive argument to reduce to the case where \(X = A = \Delta ^1 = [0,1],\) which can be shown by a direct topological argument. The inductive step is a fiber sequence argument, where we use the characterization of strong duality in terms of the Beck-Chevalley property, Proposition~\ref{efr-0029}.\begin{definition}[{Solvable pair}]\label{efr-002T}\par{}Let \(X,A\) be \hyperref[efr-002U]{topological convex spaces}. We say the pair \((X,A)\) is a \emph{solvable pair} if, for any continuous minmax problem \(L: X \times  A \to  \mathbb {R}\), strong duality holds.\end{definition}\begin{proposition}[{}]\label{efr-002V}\par{}The pair \(([0,1],[0,1])\) (in other words, \((\Delta ^1,\Delta ^1)\)) is solvable.\begin{proof}\par{}Let \(L: [0,1] \times  [0,1] \to  \mathbb {R}\) be a continuous minmax problem. Suppose strong duality does not hold. Then by adding a constant to \(L\), we can arrange that
  
  \[\sup _\theta  \inf _s L(s,\theta ) < 0 < \inf _s \sup _\theta  L(s,\theta ).\]
  
  Consider the set \(P = \{(s,\theta ) \mid  L(s,\theta ) > 0\}\). Since we must have \(\sup _\theta  L(s,\theta ) > 0\) for each \(s\), the first projection \(P \to  [0,1]\) must be surjective. Since each fiber is convex, and hence connected, and the projection \([0,1] \times  [0,1] \to  [0,1]\) is open, \(P\) is connected. As an open connected subset of a convex space, it is path connected. Hence there exists some path \(\gamma (t) \in  P\) where \(\gamma (0) = (0,\theta _0)\) and \(\gamma (1) = (1,\theta _1)\). In other words (picturing the square with the first coordinate horizontal), there exists a path from the left to the right side of the cube so that \(L(\gamma (t)) > 0\) everywhere on the path. Dually, there also exists a path from top to bottom so that \(L\) is strictly negative everywhere on that path. But they must intersect somewhere, and this is a contradiction. Hence \(L\) must have a state or a costate, finishing the proof.\end{proof}\end{proposition}
  
  \begin{proposition}[{}]\label{efr-002W}\par{}Let \(f: E \to  B\) be an affine surjection between compact Hausdorff topological convex spaces, and suppose:
  \begin{itemize}\item{}\((B,A)\) is solvable.
    \item{}For every \(b \in  B\), \((E_b,A)\) is solvable, where \(E_b \subseteq  E\) is the fiber.\end{itemize}
  Then also \((E,A)\) is solvable.
  \begin{proof}
    
    Observe first that if $L: E \times A \to \mathbb{R}$ is continuous, so is $f_!L = (b,a) \mapsto \inf_{e \in E_B}L(e,a)$. For $r \in \mathbb{R}$, consider $f_!L^{-1}((-\infty,r))$. This is simply the image $(f \times 1_A)(L^{-1}((-\infty,r)))$, which is open since a continuous surjection between compact Hausdorff spaces is always open.
  
    On the other hand, consider $f_!L^{-1}((r,\infty))$. Suppose $(b,a)$ is in this set - then there exists some $\epsilon$ so that $L(e,a) > r + \epsilon$ for $e \in E_b$.
    By continuity of $L$ there exist, for each $e \in E_b$, neighborhoods $V_e \subseteq E$ and $U_e \subset A$ so that $e \in V_e, a \in U_e$, and $L(V_e \times U_e) \geq r + \epsilon/2$. By compactness there exists a finite set of $V_e$s which cover $E_b$. Let $V$ be this union and $U$ the intersection of the corresponding $U_e$. Then $L(V \times U) \geq r + \epsilon/2$. Then $f(V) \times U$ is an open neighborhood of $(b,a)$ contained in $f_!L^{-1}((r,\infty))$, hence this set is open, hence $f_!L$ is continuous.
    
    \par{}Now recall that \((E,A)\) being solvable means the following square has the Beck-Chevalley condition for continuous \(L\):
    
    \begin{center}
        \begin{tikzcd}
            (E,*) \ar[r] \ar[d] & (*,*) \ar[d]\\
            (E,A) \ar[r] & (*,A)
        \end{tikzcd}
    \end{center} \noindent
    Now we can factor this as follows:
    \begin{center}
        \begin{tikzcd}
            (E,*) \ar[r] \ar[d] & (B,*) \ar[d] \ar[r] & (*,*)\ar[d]\\
            (E,A) \ar[r] & (B,A) \ar[r] & (*,A)
        \end{tikzcd}
    \end{center}\noindent
    By the preceding argument, and the assumption that $(B,A)$ is solvable, the right-hand square here has the Beck-Chevalley condition. So it suffices to show the left-hand square does. For a given \(L\), this means showing that these two functions on \(B\) are the same
    \[b \mapsto  \inf _{e \mapsto  b} \sup _a L(e,a)\]
    \[b \mapsto  \sup _a \inf _{e \mapsto  b} L(e,a)\]
    But this equation, for some given \(b\), is exactly strong duality in the restriction of \(L\) to \((E_b,A)\), which must hold because this is a solvable pair by assumption.
  \end{proof}

\end{proposition}

\begin{corollary}
      If \((X,[0,1])\) is solvable, so is \((X,\Delta ^n)\) for each \(n\)
    \end{corollary}
    The case $n=0$ is trivial, and the case $n=1$ is simply the hypothesis since $\Delta^1 \cong [0,1]$.
    The fibers of the map $\Delta^n \to [0,1]$ which picks out the first coordinate are all isomorphic to $\Delta^{n-1}$, except the fiber over $1$ which is simply the point, so using the proposition, we are done by induction.

\begin{lemma}[{}]\label{efr-002X}\par{}Let \(X\) be a compact topological convex space. Suppose \((X,\Delta ^n)\) is solvable for all \(n\). Then \((X,A)\) is solvable for all topological convex spaces \(A\).
    
    \begin{proof}
     Suppose for contradiction \(L: X \times  A \to  \mathbb {R}\) does not have strong duality, and assume without loss of generality \(\sup _a \inf _x L(x,a) < 0 < \inf _x \sup _a L(x,a)\). For each \(a \in  A\) let \(X_a\) consist of those \(x \in  X\)so that \(L(x,a) \leq  0\). Given some finite family \(a_0, \dots  a_n,\) consider the induced map \(a: \Delta ^n \to  A\) and apply solvability to the problem \((a_!L,X,\Delta ^n)\) - this implies in particular that \(\inf _x \sup _i L(x,a_i) \leq  0\). This means the family \(X_a\) has the finite intersection property, so by compactness it has nonempty intersection. But then an element \(x^*\) of the intersection must satisfy \(\sup _a L(x^*,a) \leq  0,\) which is a contradiction. \end{proof}
\end{lemma}
\begin{corollary}[{}]\label{efr-002Y}\par{}If \(X\) is compact, \((X,A)\) is solvable for any \(A\).
    \begin{proof}
        By Proposition~\ref{efr-002W} and its corollary, applied to Proposition~\ref{efr-002V}, we have
        \(([0,1],\Delta ^n)\) solvable for all \(n\). Since \([0,1]\) is compact, this means \(([0,1],A)\) solvable for all \(A\) by Lemma~\ref{efr-002X}. Now by duality we have \((X,[0,1])\) solvable for all \(X,\) which means \((X,\Delta ^n)\) is solvable, and by using Lemma~\ref{efr-002X} again, we are done.
    \end{proof}
\end{corollary}
\begin{proof}[Proof of Theorem~\ref{efr-001E}]
\par{}By Corollary~\ref{efr-002Y}, the pair \((X,A)\) is solvable, and strong duality holds. Since \(X,A\) are both compact, there must exist \(x^*,a^*\) attaining the infimum \(\inf _x \sup _a L(x,a)\) and the supremum \(\sup _a \inf _x L(x,a)\). These form an equilibrium.
\end{proof}
It is interesting to note the use of compactness here. Recall that topological compactness is closely connected with the property, also called compactness, of \(\operatorname {\mathrm {Hom}}(X,-)\) preserving filtered colimits (although this property, instantiated in \(\mathsf {Top}\), is not actually the same thing as topological compactness). Our use of compactness here, to derive from the existence of a state in the "finitary" subproblems \((L,X,\Delta ^n)\) the existence of a state in the entire problem, does not have this form (nor is it even the case that \(A\) is the colimit of its subsimplices), but it's possible that the proof could be rewritten to make this step more categorical.\par{}The idea of proceeding by induction on \(n\) was inspired by \cite{weinstein-elementary-minimax-2022}, although our proof is rather different - they are only looking at \emph{affine} games, and hence their induction step is completely different (and they have no need for the complicated \(n=1\) base case that we do), and since we are not merely interested in games on simplices, we need an additional compactness argument.\par{}We can use the minimax theorem to derive other statements of interest about convex optimization\begin{theorem}[{The separating hyperplane theorem (compact case)}]\label{efr-002Z}\par{}Let \(X,Y \subset  \mathbb {R}^k\) be disjoint, compact, convex subspaces. Then there exists \(v \in  \mathbb {R}^k\) and \(\alpha  \in  \mathbb {R}\) so that \(\langle  v,x \rangle  + \alpha  < 0 < \langle  v,y \rangle  + \alpha \) whenever \(x \in  X, y \in  Y\).\begin{proof}\par{}Consider the minmax problem\[(L,X \times  Y, A = \overline {B(0,1)} \subseteq  \mathbb {R}^k), L(x,y,v) = \langle  v,y-x \rangle .\]\par{}Since the closed unit ball is compact, by the minimax theorem there exists an equilibrium \(x^*,y^*,v^*\), which then satisfies
    \[\langle  v,y^*-x^* \rangle  \leq  \langle  v^*,y^*-x^* \rangle  \leq  \langle  v^*,y-x \rangle \]\par{}By disjointness, \(y^*-x^*\) must be nonzero, so with a suitable choice of \(v\) we can clearly make the left-hand item strictly positive. Hence \(\langle  v^*,y^*-x^* \rangle  =: \delta  > 0\). Now there must exist some \(\alpha  \in  \mathbb {R}\) so that \(\langle  v^*,y^* \rangle  + \alpha  = -\langle  v^*,x^* \rangle  - \alpha  = \delta  /2 > 0\).\par{}By the equilibrium property, we see that \(y^*\) must minimize \(\langle  v^*,y \rangle \) on \(Y\), and analogously \(x^*\) must maximize \(\langle  v^*,x \rangle \) on \(X\). Hence for all \(x,y\), we have
    \[\langle  v^*,x \rangle  + \alpha  \leq  -\delta  /2 < 0 < \delta  /2 \leq  \langle  v^*,y \rangle  + \alpha ,\] which concludes the proof.\end{proof}\end{theorem}
    
    \begin{theorem}[{The separating hyperplane theorem (general case)}]\label{efr-0030}\par{}Let \(X,Y \subseteq  \mathbb {R}^k\) be disjoint convex subsets. Then there exists \(v,\alpha \) so that \(\langle  v,x \rangle  + \alpha  \leq  0 \leq  \langle  v,y \rangle  + \alpha \) for all \(x \in  X, y \in  Y\).\begin{proof}\par{}Let \(K_i, L_i, i=1, \dots \) be two sequences of sets with the following properties:
    \begin{itemize}\item{}For each \(i\), \(K_i,L_i\) are disjoint.
      \item{}For each \(i\), \(K_i \subseteq  K_{i+1}\)
      \item{}Each of the \(K_i,L_i\) are compact and convex
      \item{}\(\cup _i K_i = X, \cup _i L_i = Y\)\end{itemize}
    These can be constructed for example by taking the intersection of \(X\) and \(Y\) with the boxes \([-i,i]^k\) to obtain compact, convex, disjoint subsets which exhaust \(X\) and \(Y\).\par{}Now apply Theorem~\ref{efr-002Z} to obtain a sequence of \(v_i \in  \overline {B(0,1)}\) so that \(\langle  v_i,- \rangle \) is negative on \(K_i\) and positive on \(L_i\). By compactness of the unit ball, this sequence has a point of density \(v^*\). Now for every pair \(x \in  X, y \in  Y\), we can find some \(i\) so that \(\langle  v_i,x \rangle \) is within an arbitrary \(\epsilon \) of \(\langle  v^*,x \rangle \) and the same is true for \(y\), and so that \(x \in  K_i, y \in  L_i\). But then \(\langle  v^*,y-x \rangle \) is within \(2\epsilon \) of \(\langle  v_i,y-x \rangle ,\) which is positive, so that \(\langle  v^*,y-x \rangle  \geq  0\).\par{}Now for each \(i\), \(\langle  v^*,- \rangle \) has a maximizer \(x_i^*\) on \(K_i\) and a minimizer \(y_i^*\) on \(L_i\). Hence, by an argument analogous to the proof of Theorem~\ref{efr-002Z}, there is a nonempty closed interval \([a_i,b_i]\) so that, if \(\alpha  \in  [a_i,b_i],\) we have \(\langle  v^*,- \rangle  + \alpha \) \(\leq  0\) on \(K_i\) and \(\geq  0\) on \(L_i\). But since the sets \(K_i, L_i\) are increasing this sequence of intervals must be decreasing, and hence the intersection must be nonempty - and then any \(\alpha \) in this intersection will make \(\langle  v^*,- \rangle  + \alpha \) nonpositive on \(X\), nonnegative on \(Y\), as desired.\end{proof}\end{theorem}
    
    (This is just the standard separating hyperplane theorem, see \cite[Section 2.5.1]{boyd-vandenberghe-convexopt-2004} for a textbook treatment).

    \section{The Legendre Transform}\begin{definition}[{Convex conjugate}]\label{efr-001P}\par{}Let \(V\) be a (real) vector space, and \(f: V \to  \mathbb {R}\) be a function (not necessarily linear).
    Then the \emph{convex conjugate} \(f^*: V^* \to  \mathbb {R}\) is defined by
    \[f^*(\alpha ) = \sup _x \alpha (x) - f(x)\]\end{definition}\par{}The convex conjugate is also called the \emph{Legendre transform} or the \emph{Fenchel-Legendre transform}. It is intimately related to convex duality. We will prove the following fundamental property of the convex conjugate using the categorical language of minmax problems, and along the way we will see the role that convex duality plays. Note that our invocation of the term "strong duality" here is somewhat more complicated than strictly necessary - normally one would merely invoke the separating hyperplane theorem directly.\begin{proposition}[{}]\label{efr-001Q}\par{}Let \(f: V \to  \mathbb {R}\) be convex, so that \((V,*,f)\) is a minmax problem.
    Then we can form the modified minmax problem \(L = (V,V^*,(x,\alpha ) \mapsto  f(x) - \alpha (x))\) - note that, up to a sign change in the domain, this amounts to adding the constraint \(x = 0\). Then \((L^*)^+ = -L^- = f^*\). Note that the two uses of the asterisk in this equation conflict.
    We have both the reversed optimization problem \(L^*\) given by flipping the variables,
    and the \hyperref[efr-001P]{convex conjugate} function \(f^*\). This shouldn't cause undue confusion, however. \end{proposition}\begin{proposition}[{}]\label{efr-001T}\par{}Given a minmax problem \(L = (X,Y,L)\) where \(X\) is a finite-dimensional real vector space, let
    \(L|_{0} = (X, Y \oplus  X^*, L \oplus  - \langle  -,- \rangle )\)
    Note that \((L|_0)^+(x) = \infty \) when \(x = 0\) and \(L^+(0)\) otherwise.
    Thus this amounts to adding a constraint that \(x = 0\).
    Analogously, define \(L|^0 = (L^*|_0)^* = (X \oplus  Y^*, Y, L \oplus  \langle  -, - \rangle )\). Then the Legendre transform \(f^* = ((f|_0)^*)^+\) (viewing both \(f\) and \(f^*\) as minmax problems using the inclusion \(\mathsf {Conv} \to  \mathsf {Minmax}\))\end{proposition}\begin{proposition}[{}]\label{efr-001U}\par{}Let \(f: X \to  \mathbb {R}\) be a continuous convex function defined on a vector space.
    Then there is strong duality in the minmax problem \(f|_0\)\begin{proof}\par{}Observe that \(\{x,t \mid  f(x) \leq  t\} \subseteq  X \oplus  \mathbb {R}\) is a closed convex set.
        Hence there is a hyperplane through \((0,f(0))\) so that the entire set is in one half-space.
        This means a nontrivial affine equation \(A(x,t) \geq  b\) which is satisfied whenever \(t \geq  f(x)\), and where \(A(0,f(0)) = b\).\par{}Clearly \(A(x,t) = \alpha _0(x) - at\) for some \(\alpha _0 \in  X^*, a \in  \mathbb {R}\).
        If \(a = 0\) we have \(\alpha _0(x) \leq  b\) for \emph{all} \(x\), which impossible.
        So by normalizing let's set \(a = 1\).
        This means \(\alpha _0 (x) + f(x) \geq  b = f(0)\).\par{}Recall that the minmax problem \(f|_0\) is given by \((X,X^*,L(x,\alpha ) \mapsto  f(x) - \alpha (x))\).
        Strong duality means \(\inf _x\sup _\alpha  L(x,\alpha ) = \sup _\alpha  \inf _x L(x,\alpha )\).
        We always have the inequality \(\geq \), so it suffices to identify an \(\alpha ^*\) so that
        \(\inf _x \sup _\alpha  L(x,\alpha ) \leq  \inf _x L(x,\alpha ^*)\)\par{}Clearly, for our \(L\), we have \(\inf _x\sup _\alpha  L(x,\alpha ) = f(0)\), since the supremum is \(\infty \) unless \(x = 0\).
        On the other hand, taking \(\alpha ^* = -\alpha _0\), we have \(f(0) \leq  f(x) - \alpha ^*(x)\) for all \(x\) by construction, finishing the proof.\end{proof}\end{proposition}\begin{lemma}[{}]\label{efr-003C}\par{}Given a commutative square:
  \begin{center}
    \begin {tikzcd}
    (X_1,A_1) \ar [r] \ar [d] & (Y_1,B_1) \ar [d]\\
    (X_2,A_2) \ar [r] & (Y_2,B_2)
    \end {tikzcd}
\end{center}
  \noindent in \(\mathsf {Set}^\Delta  \times  \mathsf {Set}^{\Delta ,\mathrm {op}},\) with the Beck-Chevalley property for the fibration from \(\mathsf {Minmax},\)
  let \((Z,C)\) be some other pair of convex spaces. Then the square 
  
  \begin{center}
    \begin {tikzcd}
    (X_1 \times  Z,A_1 \times  C) \ar [r] \ar [d] & (Y_1 \times  Z,B_1 \times  C) \ar [d]\\
    (X_2 \times  Z,A_2 \times  C) \ar [r] & (Y_2 \times  Z,B_2 \times  C)
    \end {tikzcd}
\end{center}\noindent
  also has the Beck-Chevalley condition\end{lemma}\begin{proposition}[{}]\label{efr-001V}\par{}Let \(f\) be a convex function.
    Then \(f = (f^*)^*\) (where \(f^*\) denotes the Legendre transform), under the identification \((X^*)^* = X\) of a finite-dimensional vector space with its double dual.\begin{proof}\par{}Recall that \(f^* = ((f|_0)^*)^+,\) as a minmax problem.
        Then the claim is that \[(((((f|_0)^*)^+)|_0)^*)^+ = f.\]
        Using first the rewrite \(((-)^*)^+ = ((-)^-)^*,\) and the notation \(((-)^*|_0)^* = -|^0,\)
        we can rewrite that as
        \[((f|_0)^-|^0)^+\]\par{}Now observe that, restricted to the subcategory \(\mathsf {Minmax}_l\) given by minmax problems \((X,Y,L)\) where \(X,Y\) are real vector spaces, and those homomorphisms given by linear (rather than merely affine) maps, \((-)|_0\) and \(-|^0\) form endofunctors, and \((-)|_0 \dashv  (-)|^0\).\par{}Since \((-)^-\) is left adjoint to the inclusion, we have \((-)|_0^- \dashv  -|^0\).
        Hence there is a canonical map, the unit of the adjunction, \(L \to  (L|_0)^-|^0\) for any \(L\).
        If \(L = (X,*,f)\) is an element of \(\mathsf {Conv}\), then by the universal property, this map factors over \(((f|_0)^-|^0)^+\). This gives us the inequality \(f \geq  (f^*)^*\).\par{}(Note that this inequality actually holds even if \(f\) is not convex, and indeed we haven't really used convexity yet).\par{}Observe that, using the natural identification \((X^*)^* = X\), we have \[(f|_0)|^0 = (X \oplus  X, X^*, (x,x';\alpha ) \mapsto  f(x) - \alpha (x) + \alpha (x')).\] Clearly \(\inf _x \sup _\alpha  f(x) - \alpha (x) + \alpha (x') = f(x'),\) since the supremum is infinite unless \(x = x'\). But observe that \((f^*)^*(x') = \sup _\alpha  \inf _x f(x) - \alpha (x) + \alpha (x')\)\par{}Our claim now is that we may exchange these extremizers by strong duality. This amounts to the claim that the local Beck-Chevalley property holds for this square at \((f|_0)|^0\):
        
  \begin{center}
    \begin {tikzcd}
            (X \oplus  (X^*)^*, *) \ar [d] \ar [r] & ((X^*)^*, *) \ar [d]\\
            (X \oplus  (X^*)^*, X^*) \ar [r] & ((X^*)^*, X^*)
            \end {tikzcd}
\end{center}\noindent
But by Proposition~\ref{efr-001U}, strong duality holds in every square of the form
        
  \begin{center}
    \begin {tikzcd}
    (X, *) \ar [d] \ar [r] & (*, *) \ar [d]\\
    (X, X^*) \ar [r] & (*, X^*)
    \end {tikzcd}
\end{center}\noindent
and by Lemma~\ref{efr-003C}, this is establishes that the previous square has the Beck-Chevalley condition as well, which finishes the proof.\end{proof}\end{proposition}

\bibliographystyle{eptcs}
\bibliography{convex}
\end{document}